\documentclass[a4paper]{amsart}

\usepackage[utf8]{inputenc}
\usepackage{mathtools}
\usepackage{amssymb,amsfonts,amsmath}
\usepackage{enumitem}
\usepackage{mathrsfs}
\usepackage{url}
\usepackage{aliascnt}
\usepackage{verbatim}
\usepackage{xcolor}
\usepackage[capitalise,nameinlink,noabbrev]{cleveref}

\numberwithin{equation}{section}
\newtheorem{theorem}{Theorem}[section]
\newaliascnt{corollary}{theorem}
\newtheorem{corollary}[corollary]{Corollary}
\aliascntresetthe{corollary}
\newaliascnt{lemma}{theorem}
\newtheorem{lemma}[lemma]{Lemma}
\aliascntresetthe{lemma}
\newaliascnt{proposition}{theorem}
\newtheorem{proposition}[proposition]{Proposition}
\aliascntresetthe{proposition}
\newtheorem{introtheorem}{Theorem}

\theoremstyle{definition}
\newaliascnt{definition}{theorem}
\newtheorem{definition}[definition]{Definition}
\aliascntresetthe{definition}
\theoremstyle{remark}
\newaliascnt{remark}{theorem}
\newtheorem{remark}[remark]{Remark}
\aliascntresetthe{remark}

\crefname{theorem}{Theorem}{Theorems}
\crefname{introtheorem}{Theorem}{Theorems}
\crefname{corollary}{Corollary}{Corollaries}
\crefname{lemma}{Lemma}{Lemmas}
\crefname{proposition}{Proposition}{Propositions}
\crefname{definition}{Definition}{Definitions}
\crefname{remark}{Remark}{Remarks}
\Crefname{theorem}{Theorem}{Theorems}
\Crefname{introtheorem}{Theorem}{Theorems}
\Crefname{corollary}{Corollary}{Corollaries}
\Crefname{lemma}{Lemma}{Lemmas}
\Crefname{proposition}{Proposition}{Propositions}
\Crefname{definition}{Definition}{Definitions}
\Crefname{remark}{Remark}{Remarks}

\newcommand{\N}{\mathbb{N}}
\newcommand{\Z}{\mathbb{Z}}

\newcommand{\T}{\mathcal{T}}
\newcommand{\OO}{\mathcal{O}}

\newcommand{\abs}[1]{\vert #1\vert}
\newcommand{\defeq}{\mathrel{\mathop{:}}=}
\newcommand\Set[2]{\{\,#1\mid#2\,\}}

\renewcommand{\epsilon}{\varepsilon}

\DeclareMathOperator{\Aut}{Aut}
\DeclareMathOperator{\Cay}{Cay}

\DeclareMathOperator{\St}{St}
\DeclareMathOperator{\Sym}{Sym}

\title[Non-uniform exponential growth]{An acylindrically hyperbolic group of non-uniform exponential growth}
\author{Roman Sauer}
\author{Eduard Schesler}
\address{Karlsruhe Institute of Technology, Englerstr.\ 2, 76131 Karlsruhe, Germany}
\email{roman.sauer@kit.edu}
\email{eduard.schesler@kit.edu}
\begin{document}

\begin{abstract}
We show the existence of an acylindrically hyperbolic group that has non-uniform exponential growth and Kazhdan's property $(T)$. The group is constructed as a common quotient of automorphism groups of free groups of increasing finite rank. 
\end{abstract}

\maketitle

\section{Introduction}

The growth of a finitely generated group records how fast balls in its Cayley
graph expand. For a finitely generated group \(G\) with a finite generating set
\(S\), let \(\gamma_G^S(\ell)\) denote the number of elements of \(G\) that can be
represented by words of length at most \(\ell\) over \(S\cup S^{-1}\).
The limit
\[
  \omega(G,S)\defeq \lim_{\ell\to\infty}\sqrt[\ell]{\gamma_G^S(\ell)}
  =\inf_{\ell\in\N}\sqrt[\ell]{\gamma_G^S(\ell)},
\]
which exists due to the submultiplicativity of \(\gamma_G^S\), is called the \emph{exponential growth rate} of \(G\) with respect to \(S\). The group \(G\) is of \emph{exponential growth} if \(\omega(G,S)>1\) for
some, and hence for every, finite generating set \(S\). Whether \(G\) has
exponential growth is independent of \(S\), but the rate \(\omega(G,S)\) itself is not.
Indeed, by enlarging \(S\) one can make \(\omega(G,S)\) arbitrarily large.
However, it is far more subtle to control how small the rates $\omega(G,S)$ can become as the generating set varies.
This motivates the study of the infimum
\[
\omega(G) := \inf_S \omega(G,S),
\]
taken over all finite generating sets of $G$.
The group \(G\) has \emph{uniform exponential growth} if \(\omega(G)>1\), and
\emph{non-uniform exponential growth} if \(G\) has exponential growth while
\(\omega(G)=1\).
The question whether these two notions differ was raised by Gromov~\cite{Gromov1981} and became highly influential.
For many naturally occurring classes of groups the answer is affirmative.
For non-elementary hyperbolic groups, uniform exponential growth was established by
Koubi~\cite{Koubi1998}. Fujiwara and Sela~\cite{FujiwaraSela23} recently showed that
the infimum defining \(\omega(G)\) is in fact attained for hyperbolic groups, and
more precisely that the set of all growth rates of a hyperbolic group is well-ordered.
For the more general class of relatively hyperbolic groups it was shown by Xie~\cite{Xie07} that these are either virtually cyclic or of uniform exponential growth.
For the even more general class of acylindrically hyperbolic groups it remained an open problem whether these can have non-uniform exponential growth, see e.g.~\cite{AbbottNgSpriano24,LegaspiSteenbock2026}.
A further property that has resisted being combined with non-uniform exponential growth is rigidity in the sense of Kazhdan. In Section~7 of his survey~\cite{delaHarpe2002}, de la Harpe asked whether there exists a group of non-uniform exponential growth that has Kazhdan's property~(T).
The main result of this article provides an affirmative answer to these questions.

\begin{introtheorem}\label{thm: main result}
There is an acylindrically hyperbolic Kazhdan group of non-uniform exponential growth. 
\end{introtheorem}    

We will construct our example as a common acylindrically hyperbolic quotient of a sequence of automorphism groups of free groups of increasing finite rank. The deduction of~\cref{thm: main result} from~\cref{thm: main estimate} is quite short and presented in~\cref{sec: conclusion}. The following~\cref{thm: main estimate} is the core of the paper. 
\begin{introtheorem}\label{thm: main estimate}
There are constants $C>0$ and $\beta\in (0,1)$ and 
finite generating sets \(T_n\) of \(\Aut(F_{4\cdot 2^n})\) such that
\[
  \gamma^{T_n}_{\Aut(F_{4\cdot 2^n})}(\ell)\le \exp\bigl(C \ell^\beta\bigr)~~\text{for every $0\le \ell\le 2^{n/2}/3$}.
\]
\end{introtheorem}

\section{Background on Grigorchuk's group \(\mathcal{G}\)}\label{sec:Grigorchuk-background}

In this section
we fix some notation for the first Grigorchuk group, which we denote by \(\mathcal{G}\), and collect the results about \(\mathcal{G}\) that will be used in the remainder of the paper.
We refer to~\cite{BartholdiGrigorchukSunic2003} for a comprehensive introduction to \(\mathcal{G}\) and to branch groups in general.
Let \(\T\) denote the binary rooted tree with vertex set \(\{0,1\}^{\ast}\) in which two vertices \(v,w\) are connected by an edge if and only if either \(v = wx\) or \(w = vx\) for some \(x \in \{0,1\}\).
For each \(n \in \N_0\), the \emph{$n$-th level} of \(\T\) is the set $X_n \defeq \{0,1\}^n$ of vertices at distance \(n\) from the root.
The group of automorphisms of \(\T\) will be denoted by \(\Aut(\T)\).
The root of \(\T\)
is the only vertex of valence \(2\) in \(\T\) and is therefore fixed by \(\Aut(\T)\).
Since moreover every element of \(\Aut(\T)\) preserves the distance between vertices, it follows that \(\Aut(\T)\) preserves each level \(X_n\).
In particular, for every subgroup \(G \leq \Aut(\T)\) we obtain a natural homomorphism from \(G\) to the symmetric group \(\Sym(X_n)\), which we denote by $\pi_n$.
We write \(\sigma_g \in \Sym(\{0,1\})\) to denote the image of \(g\) under $\pi_1$.
Conversely, we identify each permutation \(\sigma \in \Sym(\{0,1\})\) with the automorphism of \(\T\) given by \(xw \mapsto \sigma(x)w\) for all \(x \in \{0,1\}\) and \(w \in \{0,1\}^{\ast}\).
The two automorphisms of \(\T\) that arise in this way are called \emph{rooted}.
Another way to produce automorphisms of \(\T\) is to act on the two subtrees below the root independently: Each pair \((g_0, g_1) \in \Aut(\T) \times \Aut(\T)\) gives rise to the automorphism of \(\T\) given by \(xw \mapsto x\, g_x(w)\) for all \(x \in \{0,1\}\) and \(w \in \{0,1\}^{\ast}\).
Together with the rooted ones, automorphisms of this form can be used to decompose every automorphism of \(\T\) as follows.

\begin{definition}\label{def:wreath-decomposition}
Let \(\alpha \in \Aut(\T)\) and let \(v \in \{0,1\}^{\ast}\) be a vertex.
The \emph{state} of \(\alpha\) at \(v\) is the unique automorphism \(\alpha_v \in \Aut(\T)\) that satisfies
\[
\alpha(vw) = \alpha(v)\, \alpha_v(w)
\]
for every \(w \in \{0,1\}^{\ast}\).
For a vertex \(x \in \{0,1\}\) of the first level, this reads \(\alpha(xw) = \sigma_{\alpha}(x)\, \alpha_x(w)\).
The resulting decomposition \(\alpha = \sigma_{\alpha} \circ (\alpha_0, \alpha_1)\) is called the \emph{wreath decomposition} of \(\alpha\).
\end{definition}

The wreath decomposition endows us with an isomorphism
\[
\Aut(\T) \longrightarrow \Sym(\{0,1\}) \ltimes \bigl(\Aut(\T) \times \Aut(\T)\bigr),
\ \alpha \mapsto \sigma_{\alpha} \cdot (\alpha_0, \alpha_1),
\]
which we will use to identify an element $\alpha \in \Aut(\T)$ with its wreath decomposition $\sigma_{\alpha} \cdot (\alpha_0, \alpha_1)$. 

\begin{definition}\label{def:self-similar}
A subgroup \(G \leq \Aut(\T)\) is called \emph{self-similar} if for every \(g \in G\) and every vertex \(v \in \{0,1\}^{\ast}\) the state \(g_v\) is contained in \(G\).
\end{definition}

A prominent example of a self-similar group is the first Grigorchuk group \(\mathcal{G}\), see e.g.~\cite{BartholdiGrigorchukSunik03}.
It is the subgroup of \(\Aut(\T)\) generated by the four automorphisms \(a, b, c, d\) defined as follows.
The generator \(a\) is the rooted automorphism corresponding to the non-trivial permutation in \(\Sym(\{0,1\})\), while \(b, c, d\) fix the first level and are given by the wreath decompositions
\[
b = (a, c),
\qquad
c = (a, d),
\qquad
d = (1, b).
\]
A direct verification shows that each of the generators in \(S \defeq \{a, b, c, d\}\) is an involution, so that $\mathcal{G}$ is generated by $S$ as a monoid.
For each \(n \in \N_0\), we write
\[
a_n, b_n, c_n, d_n \in \Sym(X_n)
\]
to denote the images of \(a, b, c, d\) under $\pi_n$ and we set
\[
\mathcal{G}_n \defeq \langle a_n, b_n, c_n, d_n \rangle \leq \Sym(X_n).
\]
The action of \(\mathcal{G}\) on \(\T\) is \emph{spherically transitive}, that is, the induced action of \(\mathcal{G}\) on \(X_n\) is transitive for every \(n \in \N_0\); see e.g.~\cite[Section~1]{BartholdiGrigorchukSunic2003}.
In other words, the (labeled) Schreier graph of the action of \(\mathcal{G}\) on each level \(X_n\) is connected.
Recall that, given a group \(H\) acting on a set \(\Omega\) and a generating set \(S\) of \(H\), the associated \emph{labeled Schreier graph} is the graph with vertex set \(\Omega\) in which, for every \(\omega \in \Omega\) and every \(s \in S\), there is an edge labeled \(s\) from \(\omega\) to \(\omega s\).
For each \(n \in \N_0\), we write \(\Gamma_n\) for the labeled Schreier graph of the action of \(\mathcal{G}\) on \(X_n\) with respect to \(S\).
For the rest of the paper we fix the two vertices
\[
\rho_n \defeq 1^n
\qquad\text{and}\qquad
\eta_n \defeq 1^{n-1}0
\]
of \(\Gamma_n\).
As the action of \(\mathcal{G}\) on the tree \(\T\) extends continuously to its boundary \(\partial \T = \{0,1\}^{\N} = \mathfrak{C}\), we may also consider the labeled Schreier graph \(\Gamma_{\infty}\) of the orbit of the boundary point \(\rho \defeq 1^{\infty} \in \mathfrak{C}\), with respect to \(S\).
We now relate the Schreier graphs \(\Gamma_n\) and \(\Gamma_{\infty}\).
The key ingredient in doing so is the \emph{contracting property} of $\mathcal{G}$; see e.g.~\cite[Corollary~2.5]{BartholdiGrigorchukSunic2003}.
To formulate it, we write $|g|_S$ for the word length of an element $g \in \mathcal{G}$ with respect to~$S$.

\begin{lemma}\label{lem:nucleus-depth}
For every \(g \in \mathcal{G}\) and every integer
\[
n \geq \lceil \log_2 |g|_S \rceil + 1,
\]
all states \(g_v\) with \(v \in X_n\) lie in the set \(\mathcal{N} = \{1, a, b, c, d\}\).
\end{lemma}

\noindent The set \(\mathcal{N}\) in~\cref{lem:nucleus-depth} is called the \emph{nucleus} of \(\mathcal{G}\).

\begin{lemma}\label{lem:fixation-transfer}
Let \(g \in \mathcal{G}\) and let \(n \in \N\) be a number that satisfies \(n \geq \lceil \log_2 |g|_S \rceil + 2\).
Then \(g\) fixes the vertex \(\rho_n\) if and only if \(g\) fixes the boundary point \(\rho\).
\end{lemma}
\begin{proof}
If \(g\) fixes \(\rho = 1^{\infty}\), then it fixes each of its prefixes, and in particular \(\rho_n = 1^n\).
Suppose now that \(g(\rho_n) = \rho_n\) for some $n \geq \lceil \log_2 |g|_S \rceil + 2$.
By~\cref{lem:nucleus-depth}, the state \(g_{1^m}\) lies in the nucleus \(\mathcal{N} = \{1, a, b, c, d\}\) for $m = \lceil \log_2 |g|_S \rceil + 1$.
Since \(g\) fixes \(1^n\) and \(m < n\), it follows that \(g_{1^m}\) fixes \(1^{n-m}\).
As \(n - m \geq 1\), we have \(g_{1^m} \neq a\).
Hence $g_{1^m}$ lies in the set $\{1, b, c, d\}$ all of whose elements fix the ray \(\rho = 1^{\infty}\).
This gives us
\[
g(\rho) = g(1^m \rho) = 1^m\, g_{1^m}(\rho) = 1^m \rho = \rho,
\]
which completes the proof.
\end{proof}

Using~\cref{lem:fixation-transfer}, we will deduce that the balls of a certain radius $r_n$ around the points $\rho_n$ and $\eta_n$ in \(\Gamma_n\) coincide with the ball around $\rho$ in \(\Gamma_{\infty}\) of radius $r_n$, where $r_n$ goes to infinity when $n$ goes to infinity.
For a labeled graph \(G\), a vertex \(p\) of \(G\), and \(r \in \N\), we will write \(G(p, r)\) for the labeled ball of radius \(r\) around \(p\) in \(G\).
To prove the claim, we first record a general criterion under which two such balls are isomorphic.

\begin{lemma}\label{lem:stabilizer-ball-criterion}
Let \(H\) be a group with generating set \(S\) that acts on two sets \(\Omega\) and~$\Omega'$.
Let \(G\) and \(G'\) be the associated Schreier graphs in which we consider two points \(\omega \in \Omega\) and \(\omega' \in \Omega'\).
If
\[
\St_H(\omega) \cap B_S(2r) = \St_H(\omega') \cap B_S(2r),
\]
then the map
\[
\varphi \colon G(\omega, r) \rightarrow G'(\omega', r),\ \omega h \mapsto \omega' h
\]
is a well-defined isomorphism of labeled graphs.
\end{lemma}
\begin{proof}
For any \(g, g' \in H\) with \(|g|_S, |g'|_S \leq r\), the element \(g' g^{-1}\) lies in \(B_S(2r)\).
Our hypothesis therefore gives us
\[
\omega g = \omega g'
\iff
g' g^{-1} \in \St_H(\omega)
\iff
g' g^{-1} \in \St_H(\omega')
\iff
\omega' g = \omega' g'.
\]
In particular we see that \(\varphi\) is well-defined and injective.
Since every vertex of \(G'(\omega', r)\) is of the form \(\omega' g\) with \(|g|_S \leq r\) it follows that \(\varphi\) is also surjective.
Moreover \(\varphi\) sends an edge \((\omega g, \omega g s)\) of \(G(\omega, r)\) to the edge \((\omega' g, \omega' g s)\) of \(G'(\omega', r)\), which shows that \(\varphi\) preserves the labels of the edges.
\end{proof}

\begin{lemma}\label{lem:local-convergence}
For every sufficiently large \(n \in \N\) the map
\[
\Gamma_n(\rho_n, 2^{\lfloor n/2 \rfloor}) \rightarrow \Gamma_{\infty}(\rho, 2^{\lfloor n/2 \rfloor}), \ \rho_n h \mapsto \rho h
\]
is an isomorphism of labeled graphs.
\end{lemma}
\begin{proof}
Let $r_n = 2^{\lfloor n/2 \rfloor}$.
For \(n \geq 6\) we have
\[
\log_2(2 r_n) + 2 = \lfloor n/2 \rfloor + 3 \leq n.
\]
Let \(g \in \mathcal{G}\) be an element with \(|g|_S \leq 2 r_n\).
By~\cref{lem:stabilizer-ball-criterion}, it suffices to show that \(g\) fixes \(\rho_n\) if and only if it fixes \(\rho\).
Since
\[
n \geq \log_2(2r_n) + 2 \geq \lceil \log_2 |g|_S \rceil + 2,
\]
this is precisely the statement of~\cref{lem:fixation-transfer} applied to \(\rho_n\).
\end{proof}


\begin{corollary}\label{cor:eta-local-convergence}
For every sufficiently large \(n \in \N\) the map
\[
\Gamma_n(\eta_n, 2^{\lfloor n/2 \rfloor}) \rightarrow \Gamma_{\infty}(\rho, 2^{\lfloor n/2 \rfloor}), \ \eta_n h \mapsto \rho h
\]
is an isomorphism of labeled graphs.
\end{corollary}
\begin{proof}
Since $\mathcal{G}$ acts on $\T$ by automorphisms, it directly follows that \(\St_{\mathcal{G}}(\rho_n) = \St_{\mathcal{G}}(\eta_n)\).
In particular we have
\[
\St_{\mathcal{G}}(\eta_n) \cap B_S(2r_n) = \St_{\mathcal{G}}(\rho_n) \cap B_S(2r_n)
\]
for $r_n = 2^{\lfloor n/2 \rfloor}$, so~\cref{lem:stabilizer-ball-criterion} yields the isomorphism
\[
\Gamma_n(\eta_n, r_n) \to \Gamma_n(\rho_n, r_n),\ \eta_n h \mapsto \rho_n h.
\]
Composing it with the isomorphism \(\Gamma_n(\rho_n, r_n) \to \Gamma_{\infty}(\rho, r_n)\) of~\cref{lem:local-convergence} gives the claimed isomorphism.
\end{proof}

Next we show that the two vertices \(\rho_n\) and \(\eta_n\) drift apart as \(n\) tends to infinity.
In fact their distance grows exponentially.

\begin{lemma}\label{lem:separation}
The distance between \(\rho_n\) and \(\eta_n\) in \(\Gamma_n\) satisfies
\[
d_{\Gamma_n}(\rho_n, \eta_n) \geq 2^{\,n-3}
\]
for every \(n \geq 4\).
\end{lemma}
\begin{proof}
Let $n \geq 4$.
Since \(\Gamma_n\) is connected there exists an element \(g \in \mathcal{G}\) of word length \(|g|_S = d_{\Gamma_n}(\rho_n, \eta_n)\) that satisfies \(\rho_n g = \eta_n\).
Since \(1^n\) and \(1^{n-1}0\) share the prefix \(1^{n-1}\), the element \(g\) fixes the vertex \(1^{n-1}\), and its state \(g_{1^{n-1}}\) switches \(0\) and \(1\).
Let \(m \defeq \lceil \log_2 |g|_S \rceil + 1\).
By~\cref{lem:nucleus-depth}, the state \(g_{1^m}\) lies in the nucleus \(\mathcal{N} = \{1,a,b,c,d\}\).
Suppose that \(m \leq n - 2\).
Since \(g\) fixes \(1^{n-1}\), the state \(g_{1^m}\) fixes \(1^{n-1-m}\).
Using \(n-1-m \geq 1\) and the fact that \(a\) does not fix the first level, we obtain \(g_{1^m} \in \{1,b,c,d\}\).
Since each element of \(\{1,b,c,d\}\) fixes \(\rho = 1^{\infty}\) it follows that \(g_{1^m}\) fixes \(1^{n-m}\), which is a contradiction to \(g_{1^{n-1}}(1) = 0\).
Thus we have \(m \geq n-1\), which gives us
\[
\log_2 |g|_S + 1
\geq \lceil \log_2 |g|_S \rceil 
\geq n - 2.
\]
Hence \(\log_2 |g|_S \geq n - 3\), from which we deduce that $d_{\Gamma_n}(\rho_n, \eta_n) = |g|_S \geq 2^{n-3}$.
\end{proof}

We now turn to growth results concerning $\mathcal{G}$ and its Schreier graphs $\Gamma$ and $\Gamma_n$.
Let us start by recalling the following seminal result of Grigorchuk~\cite{Grigorchuk84}.

\begin{theorem}[Grigorchuk]\label{thm:G-subexponential}
The group \(\mathcal{G}\) is of subexponential growth.
More precisely, there exist constants \(C > 0\) and \(\frac{1}{2} < \gamma < 1\) such that
\[
\gamma_{\mathcal{G}}^S(n) \leq \exp(C n^{\gamma})
\]
for every $n \in \N$.
\end{theorem}

While~\cref{thm:G-subexponential} concerns the growth of balls in the Cayley graph \(\Cay(\mathcal{G},S)\), the quantity we are interested in for the Schreier graphs \(\Gamma\) and \(\Gamma_n\) is the inverted orbit growth -- a notion systematically studied by Bartholdi and Erschler~\cite{BartholdiErschler12}.
Following Bartholdi and Erschler, we will from now on work with right actions as this is convenient for the study of inverted orbits.

\begin{definition}\label{def:inverted-orbit}
Let \(H\) be a group acting from the right on a set \(\Omega\), let \(A\) be a generating set of \(H\), and let \(\omega \in \Omega\).
The \emph{inverted orbit} of a word \(w = s_1 s_2 \cdots s_{\ell} \in A^{\ast}\) at \(\omega\) is the set
\[
\OO_{\omega}(w) \defeq \{\omega,\ \omega s_{\ell},\ \omega s_{\ell-1}s_{\ell},\ \ldots,\ \omega s_1 s_2 \cdots s_{\ell}\} \subseteq \Omega.
\]
The cardinality of \(\OO_{\omega}(w)\) will be denoted by \(\delta_{\omega}(w)\).
We further define the \emph{inverted orbit growth function}
\[
\Delta_{\omega}^{A}(n) \defeq \max\Set{\delta_{\omega}(w)}{w \in A^{\ast},\ |w| \leq n},
\]
and we write
\[
N_{\omega}^{A}(n) \defeq \bigl| \Set{\OO_{\omega}(w)}{w \in A^{\ast},\ |w| \leq n} \bigr|
\]
for the number of distinct inverted orbits arising from words of length at most \(n\).
\end{definition}

Bartholdi and Erschler studied the functions $\Delta_{\rho}^{S}$ and $N_{\rho}^{S}$ for the action of \(\mathcal{G}\) on the $\mathcal{G}$-orbit of $\rho$.
In~\cite[Proposition~4.4 and Lemma~4.9]{BartholdiErschler12} they obtained the following upper bounds for these functions.

\begin{theorem}[Bartholdi-Erschler]\label{thm:BE}
There exist constants \(C > 0\) and \(0 < \alpha < 1\) such that
\[
\Delta_{\rho}^{S}(n) \leq C n^{\alpha}
\qquad\text{and}\qquad
N_{\rho}^{S}(n) \leq \exp(C n^{\alpha})
\]
for every \(n \in \N\), where \(\Delta_{\rho}^{S}\) and \(N_{\rho}^{S}\) refer to the action of \(\mathcal{G}\) on the $\mathcal{G}$-orbit of $\rho$.
\end{theorem}

Using our preceding lemmas, we transfer the estimates from~\cref{thm:BE} to the action of $\mathcal{G}$ on $\Gamma_n$ and the point $\rho_n$ and $\eta_n$.

\begin{proposition}\label{prop:BE-finite}
Let \(C\) and \(\alpha\) be as in~\cref{thm:BE} and let \(r_n = 2^{\lfloor n/2 \rfloor}\).
For every sufficiently large \(n\), every \(\xi_n \in \{\rho_n, \eta_n\}\), and every \(k \leq r_n\), we have
\[
\Delta_{\xi_n}^{S_n}(k) \leq C k^{\alpha}
\qquad\text{and}\qquad
N_{\xi_n}^{S_n}(k) \leq \exp(C k^{\alpha}),
\]
where \(\Delta_{\xi_n}^{S_n}\) and \(N_{\xi_n}^{S_n}\) refer to the action of \(\mathcal{G}_n\) on \(X_n\) with respect to \(S_n\).
\end{proposition}
\begin{proof}
We fix \(\xi_n \in \{\rho_n, \eta_n\}\) and \(k \leq r_n\).
By~\cref{lem:local-convergence} and~\cref{cor:eta-local-convergence}, the map
\[
\varphi \colon \Gamma_n(\xi_n, r_n) \rightarrow \Gamma_{\infty}(\rho, r_n),\ \xi_n h \mapsto \rho h
\]
is a well-defined isomorphism of labeled graphs.
Consider a word \(w = s_1 \ldots s_{\ell} \in S_n^{\ast}\) with \(|w| = \ell \leq k \leq r_n\).
Every element of the inverted orbit \(\OO_{\xi_n}(w)\) is of the form \(\xi_n s_i s_{i+1} \ldots s_{\ell}\), where \(s_i \ldots s_{\ell}\) is a word of length at most \(\ell \leq r_n\), and hence lies in the ball \(\Gamma_n(\xi_n, r_n)\).
The same holds for \(\OO_{\rho}(w) \subseteq \Gamma_{\infty}(\rho, r_n)\).
Under the isomorphism $\varphi$, the vertex \(\xi_n s_i \ldots s_{\ell}\) corresponds to \(\rho s_i \ldots s_{\ell}\), so the isomorphism restricts to a bijection \(\OO_{\xi_n}(w) \to \OO_{\rho}(w)\).
In particular \(\delta_{\xi_n}(w) = \delta_{\rho}(w)\), and distinct inverted orbits at \(\xi_n\) correspond to distinct inverted orbits at \(\rho\), which completes the proof.
\end{proof}

\section{The generating sets in automorphism groups of free groups}

Let \(Y\) be a finite set, and let $F(Y)$ be the free group over~$Y$. A \emph{signed basis permutation} of \(F(Y)\) is an automorphism of the form
\[
  u_y\mapsto u_{\sigma(y)}^{\epsilon_y},
  \qquad
  \sigma\in\Sym(Y),\quad \epsilon_y\in\{\pm1\}.
\]
The subgroup of $F(Y)$ generated by all signed basis permutations is called the \emph{monomial subgroup} $M(Y)$. We have \[M(Y)\cong (\Z/2)^Y\rtimes \Sym(Y).\] 
For distinct \(y,z\in Y\), let \(\lambda_{y,z}\in\Aut(F(Y))\) be the \emph{Nielsen transvection}
\[
  u_y\mapsto u_yu_z,
  \qquad
  u_{y'}\mapsto u_{y'}
  \quad\text{for }y'\neq y.
\]
The following fact is well known~\cite[Chapter I, Section 4]{LyndonSchupp01}.

\begin{theorem}\label{thm:nielsen-generating-aut-free}
For a finite set \(Y\), the group \(\Aut(F(Y))\) is generated by all signed basis permutations and all Nielsen transvections.
\end{theorem}

We set 
\[
  Y_n\defeq X_n\times \{1,2,3,4\}, 
\]
and let \(F(Y_n)\) be the free group with basis $\{u_y\mid y\in Y_n\}$. 
Thus \(F(Y_n)\cong F_{4\cdot 2^n}\). 
Next we define a finite subset $T_n\subset \Aut(F(Y_n))$
of automorphisms of $F(Y_n)$, which is partitioned into four types.  

\begin{itemize}
\item\textbf{Type 1: The truncated Grigorchuk generators.}
For every generator \(s_n\in S_n\) of $\mathcal{G}_n$, we denote -- by the same symbol -- the automorphism of \(F(Y_n)\) given by
\[
  u_{(x,i)}\mapsto u_{(x s_n,i)}
\]
for every $(x,i)\in Y_n$. 

\item\textbf{Type 2: The monomial group over \(\eta_n\).}
Let \(M_{\eta_n}(Y_n)\leq \Aut(F(Y_n))\) be the group of all signed permutations of the four basis elements
\[
  u_{(\eta_n,1)},\ldots,u_{(\eta_n,4)}
\]
and fixing all other basis elements. We call $M_{\eta_n}(Y_n)$ the \emph{monomial group over the fiber $\eta_n$}, and we have $M_{\eta_n}(Y_n)\cong (\Z/2\Z)^4\rtimes \Sym(4)$. 

\item\textbf{Type 3: The linking transposition.}
Let \(\tau_n\) be the basis transposition swapping
$u_{(\rho_n,1)}$ and $u_{(\theta_n,1)}$ 
and fixing all other basis elements.

\item\textbf{Type 4: The Nielsen transvection.}
Let \(\nu_n\) be the Nielsen transvection
\[
  \nu_n=\lambda_{(\rho_n,2),(\rho_n,3)}.
\]
\end{itemize}
\begin{definition}\label{def:Tn}
We define $
  T_n\defeq S_n\cup M_{\eta_n}(Y_n)\cup\{\tau_n,\nu_n\}
  \subseteq \Aut(F(Y_n))$.
\end{definition}

\begin{lemma}\label{lem:full-signed-monomial}
The subgroup \(\langle T_n\rangle\) contains the monomial subgroup $M(Y_n)$. 
\end{lemma}

\begin{proof}
By spherical transitivity of \(\mathcal{G}\), the group \(\mathcal{G}_n\) acts transitively on \(X_n\). Hence the fibrewise monomial subgroup $M_{x}(Y_n)$ appears as the conjugate of \(M_{\eta_n}(Y_n)\) by some word in \(S_n\). 
In particular, \(\langle T_n\rangle\) contains every signed basis transposition inside each fibre.
The graph \(\Lambda_n\) on the vertex set \(X_n\) whose edges are
\[
  \{\rho_n g,\theta_n g\},~~g\in\mathcal{G}_n
\]
is connected, which is shown in the proof of Theorem~5.4 in~\cite{SauerSchesler2026}. 

The conjugates of \(\tau_n\) by elements of \(\mathcal{G}_n\) are precisely the transpositions of the basis elements
\[
  u_{(\rho_n g,1)}
  \quad\text{and}\quad
  u_{(\theta_n g,1)}
\]
for the edges of \(\Lambda_n\). Together with the transpositions inside the fibres, these are the edge transpositions of a connected graph on the vertex set \(Y_n\). Therefore they generate \(\Sym(Y_n)\). Since inversions of basis elements are contained in the conjugates of \(M_{\eta_n}(Y_n)\), the statement follows. 
\end{proof}

\begin{theorem}\label{thm: Tn-generates-Aut}
The subset $T_n$ generates $\Aut(F(Y_n))$. 
\end{theorem}

\begin{proof}
By \cref{lem:full-signed-monomial}, the group \(\langle T_n\rangle\) contains all signed permutations of the basis. It also contains the Nielsen transvection \(\nu_n\). Conjugating \(\nu_n\) by the signed basis permutations on \(Y_n\), we obtain every Nielsen transvection. 
 Hence \(\langle T_n\rangle=\Aut(F(Y_n))\) by \cref{thm:nielsen-generating-aut-free}.
\end{proof}

\section{Proof of~\cref{thm: main estimate}}

Throughout this section we write
\[
  r_n\defeq 2^{\lfloor n/2\rfloor}.
\]

\begin{lemma}\label{lem:commuting-supports}
Let \(n\geq 5\), and let \(g,h\in S_n^\ast\) have length at most \(r_n/3\). Then every  element of \(M_{\eta_n}(Y_n)^h\) commutes with the $g$-conjugates \(\tau_n^g\) and \(\nu_n^g\) and $(\nu_n^g)^{-1}$. 
\end{lemma}

\begin{proof}
The conjugate of \(M_{\eta_n}(Y_n)\) by \(h\) is supported on the basis elements over the fibre \(\eta_n h\). The conjugate of \(\nu_n^{\pm1}\) by \(g\) is supported on the basis elements \((\rho_n g,2)\) and \((\rho_n g,3)\). The conjugate of \(\tau_n\) by \(g\) is supported on \((\rho_n g,1)\) and \((\theta_n g,1)\).

Since \(n\geq 5\), we have
\[
  1+2r_n/3<2^{n-3}.
\]
Thus it is enough to prove that \(\eta_n h\) is distinct from both \(\rho_n g\) and \(\theta_n g\). If \(\eta_n h=\rho_n g\), then
\[
  d_{\Gamma_n}(\rho_n,\eta_n)\leq \abs{g}+\abs{h}\leq 2r_n/3,
\]
contradicting \cref{lem:separation}. If \(\eta_n h=\theta_n g\), then
\[
  d_{\Gamma_n}(\rho_n,\eta_n)\leq 1+\abs{g}+\abs{h}\leq 1+2r_n/3,
\]
again contradicting \cref{lem:separation}.
\end{proof}

\begin{lemma}\label{lem:normal-form}
Let \(n\geq 5\), and let \(w\) be a word over \(T_n\cup T_n^{-1}\) of length \(\ell\leq r_n/3\). Then the element represented by \(w\) can be written as
\[
  p_1p_2p_3p_4,
\]
where \(p_1\) is represented by a word of the form \(g_1\cdots g_{\ell}\) over \(S_n\cup\{1\}\), and the remaining factors have the form
\begin{align*}
  p_2
  &=
  (a_1^{(2)})^{g_1g_2\cdots g_{\ell}}
  (a_2^{(2)})^{g_2\cdots g_{\ell}}
  \cdots
  (a_{\ell}^{(2)})^{g_{\ell}},\\
  p_3
  &=
  (a_1^{(3)})^{g_1g_2\cdots g_{\ell}}
  (a_2^{(3)})^{g_2\cdots g_{\ell}}
  \cdots
  (a_{\ell}^{(3)})^{g_{\ell}},\\
  p_4
  &=
  (a_1^{(4)})^{g_1g_2\cdots g_{\ell}}
  (a_2^{(4)})^{g_2\cdots g_{\ell}}
  \cdots
  (a_{\ell}^{(4)})^{g_{\ell}}.
\end{align*}
Here, for each \(i\), we have
\[
  g_i\in S_n\cup\{1\},
  \qquad
  a_i^{(2)}\in M_{\eta_n}(Y_n),
  \qquad
  a_i^{(3)}\in\{\tau_n,1\},
  \qquad
  a_i^{(4)}\in\{\nu_n,\nu_n^{-1},1\}.
\]
\end{lemma}

\begin{proof}
Write the word as a product of letters
\[
  t_1\cdots t_\ell
\]
and write each letter in the form
\[
  t_i=a_i^{(2)}a_i^{(3)}a_i^{(4)}g_i,
\]
where \(g_i\in S_n\cup\{1\}\), \(a_i^{(2)}\in M_{\eta_n}(Y_n)\), \(a_i^{(3)}\in\{\tau_n,1\}\), and \(a_i^{(4)}\in\{\nu_n,\nu_n^{-1},1\}\), with at most one of \(a_i^{(2)},a_i^{(3)},a_i^{(4)}\) non-trivial. This is possible because \(M_{\eta_n}(Y_n)\) is a group, the elements of \(S_n\) and \(\tau_n\) are involutions, and \(T_n\cup T_n^{-1}\) only adds \(\nu_n^{-1}\).

Pushing all \(g_i\)'s to the left gives
\[
  t_1\cdots t_\ell
  =
  g_1\cdots g_\ell
  \prod_{i=1}^{\ell}
  \bigl(a_i^{(2)}a_i^{(3)}a_i^{(4)}\bigr)^{g_i\cdots g_\ell},
\]
where \(x^g\defeq g^{-1}xg\). For \(j\in\{2,3,4\}\), call the elements
\[
  (a_i^{(j)})^{g_i\cdots g_\ell}
  \qquad (1\leq i\leq \ell)
\]
the conjugated factors of type \(j\). Thus ``type'' refers to these conjugates appearing after the rewriting, not to different generators of type \(3\) or type \(4\).

For each \(i\), the suffix \(g_i\cdots g_\ell\) has \(S_n\)-length at most \(\ell\leq r_n/3\). Hence \cref{lem:commuting-supports} implies that every conjugated factor of type \(2\) commutes with every conjugated factor of type \(3\) and with every conjugated factor of type \(4\). A conjugated factor of type \(3\) is supported on sheet \(1\), whereas a conjugated factor of type \(4\) is supported on sheets \(2\) and \(3\); hence these two factors commute as well. We may therefore move all conjugated factors of type \(2\) before all conjugated factors of type \(3\), and all conjugated factors of type \(3\) before all conjugated factors of type \(4\), preserving the order inside each type. This gives exactly the three displayed products \(p_2,p_3,p_4\), and hence the claimed normal form.
\end{proof}

\begin{proof}[Proof of~\cref{thm: main estimate}]
Let \(P_i(\ell)\) be the set of all possible factors \(p_i\) appearing in the normal form of \cref{lem:normal-form} for words of length at most \(\ell\le r_n/3\).

The first factor satisfies
\[
  \abs{P_1(\ell)}\leq \exp(C_1\ell^\gamma)
\]
by \cref{thm:G-subexponential}, after passing to the finite quotient \(\mathcal{G}_n\).

Let \(C_0>0\) be a constant such that the estimates from \cref{prop:BE-finite} hold with \(C_0\) for both \(\rho_n\) and \(\eta_n\).

Consider the factor \(p_2\) of a word of length at most~$\ell$. With the notation of \cref{lem:normal-form}, put \(u=g_1\cdots g_\ell\). Each conjugate
\[
  (a_i^{(2)})^{g_i\cdots g_\ell}
\]
is supported on the fibre over the point \(\eta_n g_i\cdots g_\ell\). Hence \(p_2\) is supported over the inverted orbit \(\OO_{\eta_n}(u)\). The number of possible inverted orbits \(\OO_{\eta_n}(u)\), with \(u\) of length at most \(\ell\), is at most \(\exp(C_0\ell^\alpha)\). For a fixed such orbit, \cref{prop:BE-finite} gives \(|\OO_{\eta_n}(u)|\leq C_0\ell^\alpha\), and at each point of the orbit the value of \(p_2\) lies in a copy of the finite group \(M_{\eta_n}(Y_n)\), of order \(2^4\cdot 4!\). Therefore
\[
  \abs{P_2(\ell)}
  \leq
  \exp(C_0\ell^\alpha)\cdot (2^4\cdot 4!)^{C_0\ell^\alpha}
  \leq
  \exp(C_2\ell^\alpha)
\]
for a suitable constant~$C_2>0$. 

The factor \(p_3\) of a word of length~$\ell$ is a permutation of the sheet-one basis elements. Each non-trivial conjugate
\[
  (a_i^{(3)})^{g_i\cdots g_\ell}
\]
is the transposition of \(u_{(\rho_n g_i\cdots g_\ell,1)}\) and \(u_{(\theta_n g_i\cdots g_\ell,1)}\). Thus the support of \(p_3\) is contained in
\[
  \{\rho_n g_i\cdots g_\ell,\ \theta_n g_i\cdots g_\ell\mid 1\leq i\leq \ell\}\times\{1\}.
\]
We now dominate this set by one inverted orbit based at \(\rho_n\). Let
\[
  \widetilde u
  :=
  a_ng_1a_n\,a_ng_2a_n\cdots a_ng_{\ell-1}a_n\,a_ng_\ell .
\]
This is a word over \(S_n\cup\{1\}\) of length at most \(3\ell\le r_n\). Since \(\theta_n=\rho_na_n\) and \(a_n^2=1\), the suffixes of \(\widetilde u\) starting at \(g_i\) give the points \(\rho_n g_i\cdots g_\ell\), while the suffixes starting at the \(a_n\) immediately before \(g_i\) give the points \(\theta_n g_i\cdots g_\ell\). Hence the support of \(p_3\) is contained in \(\OO_{\rho_n}(\widetilde u)\times\{1\}\). 
There are at most \(\exp(C_0(3\ell)^\alpha)\) possible inverted orbits \(\OO_{\rho_n}(\widetilde u)\), and each has cardinality at most \(C_0(3\ell)^\alpha\). For a fixed orbit \(O\), the factor \(p_3\) is at most an arbitrary permutation of the sheet-one basis elements over \(O\), so there are at most \(|O|!\) possibilities. Hence 
\[
  \abs{P_3(\ell)}
  \leq
  \exp(C_0(3\ell)^\alpha)\cdot \bigl(C_0(3\ell)^\alpha\bigr)!
  \leq
  \exp(C_3\ell^\alpha\log \ell)
\]
for suitable constant~$C_3>0$. 

Finally, we consider the factor \(p_4\) of a word of length~$\ell$. Setting \(u=g_1\cdots g_\ell\), the factor \(p_4\) is supported over the inverted orbit \(\OO_{\rho_n}(u)\). The number of possible such orbits is at most \(\exp(C_0\ell^\alpha)\), and each has cardinality at most \(C_0\ell^\alpha\). At a point \(x\) of this orbit the possible contribution is a power of the transvection
\[
  u_{(x,2)}\mapsto u_{(x,2)}u_{(x,3)}
\]
with exponent between \(-\ell\) and \(\ell\), so there are at most \(2\ell+1\) choices at each point. Thus, 
\[
  \abs{P_4(\ell)}
  \leq
  \exp(C_0\ell^\alpha)\cdot (2\ell+1)^{C_0\ell^\alpha}
  \leq
  \exp(C_4\ell^\alpha\log \ell)
\]
for a suitable constant \(C_4>0\).

Choose \(\beta\in(\max\{\alpha,\gamma\},1)\). Because of \(\ell^\alpha\log\ell\leq C_\beta\ell^\beta\) the product
\[
  \abs{P_1(\ell)}\abs{P_2(\ell)}\abs{P_3(\ell)}\abs{P_4(\ell)}
\]
is bounded above by \(\exp(C\ell^\beta)\), after increasing \(C\). This product bounds the \(\ell\)-ball with respect to \(T_n\), so \cref{thm: main estimate} follows.
\end{proof}

\section{Proof of~\cref{thm: main result}}\label{sec: conclusion}

The following statement is a direct consequence of the fact that the automorphism group of a free group of rank at least~2 is acylindrically hyperbolic, due to Genevois--Horbez~\cite[Corollary~1.2]{GenevoisHorbez2020}, and the common quotient theorem of Minasyan--Osin~\cite[Theorem~1.1]{MinasyanOsin2018}. The latter says that every countable family of countable acylindrically hyperbolic groups has a common finitely generated acylindrically hyperbolic quotient.

\begin{theorem}\label{thm: common quotient}
The automorphism groups of free groups $\Aut(F_n)$, $n\ge 2$, have a common finitely generated acylindrically hyperbolic quotient. 
\end{theorem}

\begin{proof}[Proof of~\cref{thm: main result}]
By~\cref{thm: common quotient} there is a finitely generated acylindrically hyperbolic group \(Q\) and epimorphisms
\[
  q_n\colon \Aut(F_{4\cdot 2^n})\twoheadrightarrow Q
\]
for every $n\in\N$. The group \(Q\) is a Kazhdan group since it is a quotient of Kazhdan groups by the result of Kaluba-Kielak-Nowak~\cite[Theorem~1.1]{KalubaKielakNowak2021} or the earlier result of Kaluba--Nowak--Ozawa~\cite{KalubaNowakOzawa2018}. See also~\cref{rem: property T} below. 
Let \(U_n=q_n(T_n)\) be the image of the generating set \(T_n\) from~\cref{thm: main estimate}. Taking quotients cannot increase growth. Thus, by~\cref{thm: main estimate}, there are constants \(C>0\) and \(\beta\in (0,1)\) such that
\[
  \gamma_Q^{U_n}(\ell)
  \le
  \gamma_{\Aut(F_{4\cdot 2^n})}^{U_n}(\ell)
  \le
  \exp\bigl(C\ell^\beta\bigr)
  \qquad
  \text{for every }0\leq \ell\leq 2^{n/2}/3.
\]
Set \(\ell_n\defeq\lfloor 2^{n/2}\rfloor/3\). Then
\begin{align*}
  1\leq \omega(Q)
  &\leq \inf_{n\in\N}\omega(Q,U_n)\\
  &= \inf_{n\in\N}\inf_{\ell\geq 1}\gamma_Q^{U_n}(\ell)^{1/\ell}\\
  &\leq \inf_{n\in\N}\inf_{1\leq\ell\leq 2^{n/2}}\gamma_Q^{U_n}(\ell)^{1/\ell}\\
  &\leq \inf_{n\in\N}\inf_{1\leq\ell\leq 2^{n/2}}\exp\bigl(C\ell^{\beta-1}\bigr)\\
  &\leq \inf_{n\in\N}\exp\bigl(C\ell_n^{\beta-1}\bigr)
   =1.
\end{align*}
Thus \(\omega(Q)=1\).
On the other hand, it is well known that \(Q\) has exponential growth since an acylindrically hyperbolic group contains 
a non-abelian free group (coming from independent loxodromic elements)~\cite{Osin16}. 
So \(Q\) has non-uniform exponential growth.
\end{proof}

\begin{remark}\label{rem: property T}
To obtain the Kazhdan property for~$Q$ one does not have to appeal to results about the Kazhdan property for automorphism groups of free groups. One could simply take the common quotient of $Q$ as above and a Kazhdan hyperbolic group. 
\end{remark}

\bibliographystyle{alpha}
\IfFileExists{literature.bib}
  {\bibliography{literature.bib}}
  {\bibliography{../literature.bib}}

\end{document}